\documentclass[11pt]{amsart}
\usepackage{amssymb}

\theoremstyle{plain}
\newtheorem{theorem}{Theorem}

\newtheorem{lemma}{Lemma}

\theoremstyle{definition}

\theoremstyle{remark}

\numberwithin{equation}{section}

\begin{document}
\title[Polynomial Solutions to Pell's Equation ]
       {Polynomial Solutions to Pell's Equation and Fundamental Units in Real Quadratic Fields}
\author{J. Mc Laughlin}
\address{Mathematics Department\\
       University of Illinois \\
        Champaign - Urbana, Illinois 61820}
\email{jgmclaug@math.uiuc.edu}
\date{January,21,1999}

\begin{abstract}
Finding polynomial  solutions to Pell's equation is of interest as
such solutions sometimes allow the fundamental units to be
determined in an infinite class of real quadratic fields.

In this paper,
for each triple of positive integers $(c,h,f)$ satisfying
\[c^{2}-f\,h^{2}=1,
\]
where $(c,h)$ are the smallest pair of integers satisfying this
equation, several sets of polynomials $(c(t),h(t),f(t))$ which satisfy
\[c(t)^{2}-f(t)\,h(t)^{2}=1 \text{ and } (c(0),h(0),f(0)) = (c,h,f)
\]
 are derived.
Moreover, it is shown that  the pair $(c(t),h(t))$
constitute the fundamental polynomial solution to the Pell's equation above.

The continued fraction expansion of $\sqrt{f(t)}$ is
 given in certain general cases (for example, when the continued
 fraction expansion of $\sqrt{f}$ has odd period length, or even
 period length or has period length $\equiv 2 \mod{4}$ and the middle
 quotient has a particular form etc).

Some applications to determining the fundamental unit in real
quadratic fields is also discussed.
\end{abstract}

\maketitle

\section{Introduction: Polynomial Solution's to Pell's Equation}\label{S:intro}
Finding polynomial solutions to Pell's equation is of interest as such solutions
sometimes allow the fundamental units to be determined in an infinite class of
real quadratic fields. Finding such polynomial solutions is not only an
interesting problem in its own right, but knowing the fundamental unit in
a real quadratic field means that the class number of the field can be found
via some version of  Dirichlet's class number formula (see, for example,
\cite{R01}).

In \cite{oP13}, Perron gives some examples of  polynomials $f$ for
which the continued fraction expansion of $\sqrt{f}$ can be stated
explicitly, but the period in these cases is at most 6.
These examples were added to in a paper by Yamamoto (\cite{Y71}) and later, in
\cite{B76}, Bernstein obtained a large number of such polynomials $f$ for which
the continued fraction expansion of  $\sqrt{f}$ could be made arbitrarily
long and in \cite{B76a} he gave explicit expressions for the
fundamental unit in the quadratic field $\mathbb{Q}(\sqrt{f})$.
Further examples were given in papers by Levesque and Rhin~\cite{LR86},
 Madden~\cite{M01}, Van Der Poorten~\cite{VDP94} and
 Van Der Poorten and Williams~\cite{VDPW99}.

In this paper,
for each triple of positive integers $(c,h,f)$ satisfying
\[c^{2}-f\,h^{2}=1,
\]
where $(c,h)$ are the smallest pair of integers satisfying this
equation, several sets of polynomials $(c(t),h(t),f(t))$ which satisfy
\[c(t)^{2}-f(t)\,h(t)^{2}=1 \text{ and } (c(0),h(0),f(0)) = (c,h,f)
\]
 are derived.
Moreover, it is shown that  the pair $(c(t),h(t))$
constitute the fundamental polynomial solution to the Pell's equation above.

The continued fraction expansion of $\sqrt{f(t)}$ is
 given in certain general cases (for example, when the continued
 fraction expansion of $\sqrt{f}$ has odd period length, or even
 period length or has period length $\equiv 2 \mod{4}$ and the middle
 quotient has a particular form etc).

Some applications to determining the fundamental unit in real
quadratic fields are also discussed.

\textbf{Definition:} A polynomial $f(t) \in Z[\,t]\,$  is called a \emph{Fermat-Pell polynomial in one
variable} if there exists
polynomials $c(t), h(t) \in Z[\,t\,]\,$
such that
\begin{equation}\label{E:pe}
c(t)^{2} - f(t) h(t)^{2} = 1,\,\, \text{for all t}.
\end{equation}

A triple of polynomials $(c(t), h(t), f(t))$ satisfying
equation~\ref{E:pe} constitutes a \emph{ polynomial solution}
to Pell's equation.
A Fermat-Pell polynomial is said
to have a \emph{polynomial continued fraction expansion} if there
exists an integer $T$ and  a positive integer
$n$ such that
\begin{equation} \label{E:pcf}
\sqrt{f(t)} = [\,a_{0}(t);\overline{a_{1}(t), \cdots ,
a_{n}(t)} ]\,,\,
\end{equation}
for all integral $t \geq T$, where the $a_{i}(t) \in \mathbb{Z} [\,t\,]$. For a fixed Fermat-Pell
polynomial $f(t)$ a pair of polynomials $(c(t),h(t))$ is said to be the
\emph{fundamental polynomial solution} to ~\eqref{E:pe} if
$(c(t),h(t))$ constitutes the smallest solution in positive integers to
equation~\ref{E:pe}
for all non-negative integral values of $t \geq T$.

 In this paper
it is shown how to construct several infinite families of
 Fermat-Pell polynomials in which the continued fraction expansion
of the polynomials can have arbitrary long  period length.
Moreover, the polynomial continued fraction expansion of
these polynomials can be written down explicitly.


\section{Notation/ useful facts about convergents}
As usual $[\,a_{0};\overline{a_{1} ,\cdots , a_{n},2a_{0}}]$ will
denote the simple infinite periodic continued fraction
{\allowdisplaybreaks
\begin{equation*}
a_{0} + \cfrac{1}
          {a_{1} + \cfrac{1}
                 {  \ddots a_{n-1}+ \cfrac{\ddots}
                           {a_{n} + \cfrac{1}
                                  {2a_{0}  + \cfrac{1}
                                               {a_{1} +  \ddots}
                                          }
                                 }
                          }
                }.
\end{equation*}
}
$\displaystyle{\frac{A_{i}}{B_{i}}}$ will denote the $i$'th convergent:
\[
a_{0} + \cfrac{1}{a_{1} + \cfrac{1}{a_{2} + \cdots
\cfrac{1}{a_{i} }}}.
\]

Use will also be made of the following:
\begin{align}\label{E:recur}
&A_{i} = a_{i}A_{i-1} + A_{i-2}, \\
&B_{i} = a_{i}B_{i-1} + B_{i-2},\notag\\
&A_{i}B_{i-1} - A_{i-1}B_{i} = (-1)^{i-1}, \text{ where }\notag\\
 &A_{-1}= B_{-2}=1, \,\,A_{-2}= B_{-1}=0,\notag
\end{align}
each of these being valid for $i = 0,1,2 \cdots$.

\section{ Some necessary Lemmas}

Before coming to the main results of the paper
several lemmas from the theory of the convergents of the
 continued fraction  expansion of $\sqrt{f}$, where $f$ is a
non-square positive integer, are needed.
Some are given without proof since the results are well known or are
straightforward. In what follows $c$ and $h$ will denote the smallest pair
of positive integers satisfying $c^{2}-fh^{2}=1$. Recall that if the
 continued fraction  expansion of
$\sqrt{f}$ has period $2m$ then $(c,h) = (A_{2m-1},B_{2m-1})$ and that
if the continued fraction  expansion of
$\sqrt{f}$ has period $2m+1$ then $(c,h) = (A_{4m+1},B_{4m+1})$.


\begin{lemma}\label{L:cfinv}
Denote the $i$'th convergent to the continued fraction \\
$b_{0} + \displaystyle{\frac{1}{b_{1}+}\frac{1}{b_{2}+ \ldots}}$ by
$b_{0} + \displaystyle{\frac{1}{b_{1}+}\frac{1}{b_{2}+} \cdots +\frac{1}{b_{i}}} =
\frac{P_{i}}{Q_{i}}$. Then
\begin{equation}\label{E:invcf}
b_{m} + \frac{1}{b_{m-1}+}\frac{1}{b_{m-2}+} \cdots +\frac{1}{b_{0}} =
\frac{P_{m}}{P_{m-1}},\,\, \text{for all}\,\, m \geq 1 .
\end{equation}
\end{lemma}
\begin{proof}
This is well known.
\end{proof}


Label the usual three sequences of positive integers used to determine
the continued fraction expansion $[a_{0};a_{1}, \cdots]$ of $\sqrt{f}$ as
follows:\\
$r_{0} = 0, s_{0} = 1, a_{0} = \lfloor \sqrt{f} \rfloor$,
$r_{k+1} = a_{k}s_{k} - r_{k}$,
$s_{k+1} = \displaystyle{(f - r_{k+1}^{2})/s_{k}}$ and
$a_{k+1} = \displaystyle{\lfloor (\sqrt{f} +
r_{k+1})/s_{k+1}\rfloor}$, for
 $k \geq 0$.
\begin{lemma}
\label{L:rsa}
\begin{equation}\label{E:con1}
A_{k}A_{k-1} - fB_{k}B_{k-1} = (-1)^{k}r_{k+1},\text{ for }k \geq -1.
\end{equation}
\begin{equation}\label{E:con2}
A_{k}^{2} - fB_{k}^{2} = (-1)^{k+1}s_{k+1},\text{ for } k \geq -1.
\end{equation}
\end{lemma}

\begin{proof}
A simple double induction.
\end{proof}


\begin{lemma}\label{L:rk}
$r_{k} \leq a_{0}$,\, for all $k \geq 0.$
\end{lemma}

\begin{proof}
$r_{k+1} = a_{k}s_{k} - r_{k} =
\displaystyle{\lfloor (\sqrt{f} + r_{k})/s_{k}\rfloor}s_{k} -
r_{k}
<( \displaystyle{(\sqrt{f} + r_{k})/s_{k}})s_{k} - r_{k}\\
= \sqrt{f} < a_{0} + 1.$
\end{proof}


\begin{lemma} \label{L:sm}
If the continued fraction expansion of $\sqrt{f}$ has
period 2m and $a_{m} = a_{0}$ or $a_{0} - 1$ then $s_{m} = 2$.
\end{lemma}

\begin{proof}
By Lemma \ref{L:rk} and the definition of
$a_{0}$,
$a_{m} = \displaystyle{\lfloor (\sqrt{f} + r_{m})/s_{m}\rfloor}$
 $<\,\displaystyle{ (\sqrt{f} + r_{m})/s_{m}}
 <\,\displaystyle{ (2a_{0} + 1)/s_{m}}$.
If $a_{m} = a_{0}$ then
$a_{0} < \,\displaystyle{ (2a_{0} + 1)/s_{m}}$ and the fact that
$r_{m} \ne 0$ clearly imply
$s_{m} = 2$ (and $r_{m} = a_{0}$).
The case for $a_{m} = a_{0} - 1$ is similar (with $r_{m} = a_{0} -
1$).
\end{proof}


\begin{lemma}\label{L:sol}
If the continued fraction expansion of $\sqrt{f}$ has period of even
length, say equal to $2m$,  then
\[
c = A_{m}B_{m-1} + A_{m-1}B_{m-2},\,\,\,\, h = B_{m-1}(B_{m} +
B_{m-2}).
\]
\end{lemma}

\begin{proof}
By the symmetry of the continued fraction expansion of $\sqrt{f}$
\[
[a_{1},a_{2},\cdots,a_{2m-2},a_{2m-1}] =
[a_{1},a_{2},\cdots,a_{m-1},a_{m},a_{m-1},\cdots,a_{2},a_{1}].
\]
We know that
$c = A_{2m-1} ,\,\,\, h = B_{2m-1}$ and by the correspondence between
matrices and convergents:
\allowdisplaybreaks{
\begin{align*}
\left(
\begin{matrix}
a_{0} & 1 \\
1     & 0
\end{matrix}
\right)
\left(
\begin{matrix}
a_{1} & 1 \\
1     & 0
\end{matrix}
\right)
\cdots
\left(
\begin{matrix}
a_{2m-1} & 1 \\
1     & 0
\end{matrix}
\right)
&=\left(
\begin{matrix}
A_{2m-1} & A_{2m-2} \\
B_{2m-1} & B_{2m-2}
\end{matrix}
\right),
\\
\left(
\begin{matrix}
a_{0} & 1 \\
1     & 0
\end{matrix}
\right)
\left(
\begin{matrix}
a_{1} & 1 \\
1     & 0
\end{matrix}
\right)
\cdots
\left(
\begin{matrix}
a_{m} & 1 \\
1     & 0
\end{matrix}
\right)
&=
\left(
\begin{matrix}
A_{m} & A_{m-1} \\
B_{m} & B_{m-1}
\end{matrix}
\right),\\
\left(
\begin{matrix}
a_{m-1} & 1 \\
1     & 0
\end{matrix}
\right)
\cdots
\left(
\begin{matrix}
a_{1} & 1 \\
1     & 0
\end{matrix}
\right)
&=
\left(
\begin{matrix}
B_{m-1} & A_{m-1} - a_{0}B_{m-1}\\
B_{m-2} & A_{m-2} - a_{0}B_{m-2}
\end{matrix}
\right),
\end{align*}}
\[\text{ so that }
\left(
\begin{matrix}
A_{2m-1} & A_{2m-2} \\
B_{2m-1} & B_{2m-2}
\end{matrix}
\right)
=
\left(
\begin{matrix}
A_{m} & A_{m-1} \\
B_{m} & B_{m-1}
\end{matrix}
\right)
\left(
\begin{matrix}
B_{m-1} & A_{m-1} - a_{0}B_{m-1}\\
B_{m-2} & A_{m-2} - a_{0}B_{m-2}
\end{matrix}
\right)
\]
and the result follows.
\end{proof}

\textbf{Remark:} One can show in essentially the same way that if the
period of the continued fraction expansion of$\sqrt{f}$ is odd, say $2m
+ 1$,
then
\begin{align*}
c = A_{4m+1} &= (A_{m}^{2} + A_{m-1}^{2})(B_{m}^{2} + B_{m-1}^{2})
+ (A_{m}B_{m} + A_{m-1}B_{m-1})^{2}\\
&= 2(A_{m}B_{m} + A_{m-1}B_{m-1})^{2} + 1 \,\,\,\mbox{and}\\
h = B_{4m+1} &= 2(A_{m}B_{m} + A_{m-1}B_{m-1})(B_{m}^{2} +
B_{m-1}^{2}).
\end{align*}


\begin{lemma}\label{L:amcb}
If the continued fraction expansion of $\sqrt{f}$ has period of even
length which is congruent to $2 \mod{4}$, say equal to $2m$,
 then
\[
hA_{m-1} - (c - 1)B_{m-1} = 0.
\]
\end{lemma}
\begin{proof} By Lemma~\ref{L:sol}:
\begin{align*}
&hA_{m-1} - (c - 1)B_{m-1} = B_{m-1}(B_{m} + B_{m-2})A_{m-1}\\
&\phantom{assaasdsasaddsadsadsa}-
(A_{m}B_{m-1} + A_{m-1}B_{m-2} - 1)B_{m-1}
\hspace*{60pt} \\
&= B_{m-1}(a_{m}B_{m-1} + B_{m-2} + B_{m-2})A_{m-1}\\
&\phantom{assadsadsadsa}-
((a_{m}A_{m-1} + A_{m-2})B_{m-1} + A_{m-1}B_{m-2} - 1)B_{m-1}\\
&= 2A_{m-1}B_{m-1}B_{m-2} - (A_{m-2}B_{m-1}^{2} + A_{m-1}B_{m-1}B_{m-2}
- B_{m-1})\\
&= B_{m-1}(A_{m-1}B_{m-2} - A_{m-2}B_{m-1} + 1) = B_{m-1}((-1)^{m-2} +
1) = 0.
\end{align*}
\end{proof}


\begin{lemma}\label{L:azeo}
If $\sqrt{f} = [\,a_{0};\overline{a_{1} , \cdots, a_{n},2a_{0}}]\,$
then $A_{n} = a_{0}B_{n} + B_{n-1}$.
\end{lemma}

\begin{proof}
Case(1) (the period is even, $= 2m$, say): By Lemma~\ref{L:sol},
\begin{align*}A_{n}  &=   A_{m}B_{m-1}  +    A_{m-1}B_{m-2},\,\,
B_{n} = B_{m-1}(B_{m} + B_{m-2})\text{ and }  \\
B_{n-1} &= B_{m}(A_{m-1} - a_{0}B_{m-1}) + B_{m-1}(A_{m-2} -
a_{0}B_{m-2})
\end{align*}
 and the result follows easily.

\vspace{12pt}

Case(2) (The period is odd): Suppose the period is odd, $= 2m + 1$, say and \\
$\sqrt{f} = [\,a_{0};
\overline{a_{1} , \cdots , a_{m},a_{m},\cdots, a_{1},2a_{0}}].\,$ Then
\begin{equation*}
\left(
\begin{matrix}
A_{n} & A_{n-1} \\
B_{n} & B_{n-1}
\end{matrix}
\right)
=
\left(
\begin{matrix}
A_{m} & A_{m-1} \\
B_{m} & B_{m-1}
\end{matrix}
\right)
\left(
\begin{matrix}
B_{m} & A_{m} - a_{0}B_{m}\\
B_{m-1} & A_{m-1} - a_{0}B_{m-1}
\end{matrix}
\right)
\end{equation*}
$A_{n} = A_{m}B_{m} + A_{m-1}B_{m-1},\,\,\,\,\,
 B_{n} = B_{m}^{2} + B_{m -1}^{2}$
, \\
$B_{n-1} =  B_{m}(A_{m} - a_{0}B_{m}) + B_{m-1}(A_{m-1} -
a_{0}B_{m-1})$
and again the result follows by simple arithmetic.
\end{proof}


\begin{lemma}\label{L:bnf}
If $\sqrt{f} = [\,a_{0};\overline{a_{1} ,\cdots , a_{n},2a_{0}}]\,$
then $B_{n}(f - a_{0}^{2}) = A_{n-1} + a_{0}B_{n-1}$.
\end{lemma}

\begin{proof}
Recall that if the period is even ( $n$ odd) then $A_{n}^{2} - f
B_{n}^{2} = 1$, that if the period is odd ($n$ even) then $A_{n}^{2} - f
B_{n}^{2} = -1$ and that $A_{n}B_{n-1} - A_{n-1}B_{n} =
(-1)^{n-1}$.
Then
\begin{align*}
&B_{n}(f - a_{0}^{2}) = A_{n-1} + a_{0}B_{n-1}\\
&\Longleftrightarrow B_{n}f - a_{0}(A_{n} - B_{n-1}) = A_{n-1} +
a_{0}B_{n-1},\,\,\mbox{ (by  Lemma~\ref{L:azeo})}\\
&\Longleftrightarrow B_{n}f = a_{0}A_{n} + A_{n-1}\\
&\Longleftrightarrow B_{n}^{2}f = a_{0}A_{n}B_{n} + A_{n-1}B_{n}\\
&\Longleftrightarrow A_{n}^{2} + (-1)^{n} = a_{0}A_{n}B_{n} +
A_{n}B_{n-1} + (-1)^{n}\\
&\Longleftrightarrow A_{n} = a_{0}B_{n} + B_{n-1}
\end{align*}
 and the result is
true by Lemma~\ref{L:azeo}.
\end{proof}


\begin{lemma}\label{L:ansq}
If the  continued fraction  expansion of $\sqrt{f}$ has period of odd length, $ = 2m +
1$ , say,
then $c = 2A_{2m}^{2} + 1$ and $h= 2A_{2m}B_{2m}$ .
\end{lemma}
\begin{proof}
Elementary.
\end{proof}


\begin{lemma}\label{L:trbl}
Let
 $\sqrt{f}= [\,a_{0};
\overline{a_{1} ,\cdots, a_{m-1},a_{m},a_{m-1},\cdots, a_{1},2a_{0}}]\, $
, where
$a_{m} = a_{0}$ or $a_{0} - 1$. Then
\begin{align*}
&(i) A_{m-1} = B_{m} + B_{m-2}\\
&(ii)\, c - 1 = A_{m-1}(B_{m} + B_{m-2}), \mbox{ if $m$ is odd,}\\
&(iii)\, c + 1 = A_{m-1}(B_{m} + B_{m-2}), \mbox{ if $m$ is even}.
\end{align*}
\end{lemma}
\begin{proof}
(i)By the symmetry of the sequence $\{\,r_{i}\,\}$ ($r_{m+i} = r_{m-i+1}$),
it follows that  $
 r_{m+1}
= r_{m}.$
By Lemma~\ref{L:rsa} and Lemma~\ref{L:sm}
\begin{align*}
&A_{m}A_{m-1} - fB_{m}B_{m-1} = (-1)^{m}r_{m+1} =(-1)^{m} a_{m}\\
&A_{m-1}^{2} - fB_{m-1}^{2} = (-1)^{m}s_{m} =2 (-1)^{m}  \\
&\Longrightarrow  f = \frac{A_{m-1}^{2} + (-1)^{m-1}2}{B_{m-1}^{2}} =
\frac{A_{m}A_{m-1} + (-1)^{m-1} a_{m}}{B_{m}B_{m-1}} \\
&\Longrightarrow A_{m-1}^{2}B_{m} + (-1)^{m-1} 2B_{m} =
A_{m}A_{m-1}B_{m-1} + (-1)^{m-1} a_{m}B_{m-1} \\
&\Longrightarrow A_{m-1}^{2}B_{m} + 2B_{m} =
A_{m-1}(A_{m-1}B_{m}+ (-1)^{m-1} ) + a_{m}B_{m-1} \\
&\Longrightarrow 2B_{m} = A_{m-1}+ a_{m}B_{m-1}
\end{align*}
The result follows from
the recurrence relation for the $B_{i}'s$.\\
(ii)By Lemma~\ref{L:sol}:
\begin{align*}
c - 1 & = A_{m}B_{m-1} + A_{m-1}B_{m-2} - 1\\
& = A_{m-1}B_{m} +
(-1)^{m-1} + A_{m-1}B_{m-2} - 1\\
&= A_{m-1}B_{m} + A_{m-1}B_{m-2}.
\end{align*}
(iii)Again by Lemma~\ref{L:sol}:
\begin{align*}
c + 1 & = A_{m}B_{m-1} + A_{m-1}B_{m-2} + 1\\
& = A_{m-1}B_{m} +
(-1)^{m-1} + A_{m-1}B_{m-2} + 1\\
&= A_{m-1}B_{m} + A_{m-1}B_{m-2}.
\end{align*}
\end{proof}


\begin{lemma}\label{L:smasol}
Let  $\sqrt{f} = [\,a_{0};\overline{a_{1} ,\cdots ,
a_{n},2a_{0}}]\,$.
If $X$ and $Y$ are positive integers satisfying
\[X^{2}-fY^{2}= \pm 1
\]
then $(X,Y)= (A_{k(n+1)-1},B_{k(n+1)-1})$, for some positive integer $k$.
\end{lemma}
\begin{proof}
This is well known (See, for example,~\cite{B89}, page 387).
\end{proof}
The way this lemma will be used is as follows: Suppose it is known
that $X^{2}-fY^{2}= \pm 1$ and that the continued
fraction expansion of $X/Y$ is $[\,a_{0};a_{1} , \cdots,
a_{n}]\,$. By the lemma, $X/Y = A_{k(n+1)-1}/B_{k(n+1)-1}$,
for some positive integer $k$,\, and thus that the finite continued fraction
expansion $[\,a_{0};a_{1} , \cdots,a_{n}]\,$ either
contains at least one complete period of the continued fraction
expansion of $\sqrt{f}$ (the case $k>1$) or else is just one term short of a
full period (the case $k=1$), in which case the missing term is of
course $2a_{0}$ and $\sqrt{f}=[\,a_{0};\overline{a_{1} ,\cdots ,
a_{n},2a_{0}}]\,$. In particular, if $ a_{i} \not = 2a_{0}$ for some
$i,\, 1 \leq i \leq n$, then the sequence $a_{0}, a_{1} ,\cdots ,
a_{n}$ cannot contain a full period in the continued fraction expansion
of $\sqrt{f}$ (since a full period of course ends in the term $2\,a_{0}$).
This implies that $k=1$ and that the latter case holds.


In the following theorems
$\sqrt{f} = [\,a_{0};\overline{a_{1} ,\cdots , a_{n},2a_{0}}]\,$,
unless otherwise stated. As above, $c$ and $h$ are the smallest
pair of positive integers satisfying $c^{2}-fh^{2}=1$.
The variable $t$ will be
assumed to be a real variable. Initially, in the theorems that follow, the
results shall be shown to be true for non-negative integral $t$ and follow for
real $t$ by continuation.


\allowdisplaybreaks{\begin{theorem}\label{T:hto}
Let $f(t)=h^{2}t^{2} + 2ct + f$.\\
(i)If $n$ is even then
\begin{equation*}
\sqrt{f(t)} =
[\,ht + a_{0};\overline{a_{1} ,\cdots , a_{n},
2a_{0},a_{1} ,\cdots , a_{n},2(ht + a_{0})}],\,\,\,\,
 \text{for all } t \, \geq
0.
\end{equation*}
(ii)
If $n$ is odd then
\begin{equation*}
\sqrt{ f(t)} =
[\,ht + a_{0};\overline{a_{1} , \cdots, a_{n}
,2(ht + a_{0})}]\,,\, \mbox{ for all } t \, \geq
0.
\end{equation*}
(iii)In either case $X =h^{2}t+c,\,\,\,Y = h$ constitute the fundamental
solution to $X^{2} - f(t)Y^{2}=1$.
\end{theorem}}

\begin{proof}
It can be easily checked that $(h^{2}t+c)^{2}-f(t)h^{2}=1$.\\
(i)For $n$ even
\begin{align*}
&[\,ht + a_{0};a_{1} , \cdots, a_{n},
2a_{0},a_{1} , \cdots, a_{n}]
= [\,ht + a_{0};a_{1} , \cdots, a_{n},
a_{0}+\frac{A_{n}}{B_{n}}]\\
&= ht +\frac{\left(a_{0}+\frac{A_{n}}{B_{n}}\right)A_{n}+A_{n-1}}
        {\left(a_{0}+\frac{A_{n}}{B_{n}}\right)B_{n}+B_{n-1}}
= ht +\frac{\left(a_{0}B_{n}+A_{n}\right)A_{n}+A_{n-1}B_{n}}
        {\left(a_{0}B_{n}+A_{n}\right)B_{n}+B_{n-1}B_{n}}\\
&= ht +\frac{2A_{n}^{2}-1}
        {2A_{n}B_{n}},\,\mbox{ by Lemma~\ref{L:azeo} and
~\eqref{E:recur}},\\
&= ht +\frac{c}{h} = \frac{h^{2}t+c}{h},\,\mbox{ by
Lemma~\ref{L:ansq}}.
\end{align*}
Since $2(ht+a_{0})$ is not in the sequence $\{a_{1} ,\cdots , a_{n},
2a_{0},a_{1} ,\cdots , a_{n}\}$ it follows from Lemma\ref{L:smasol} that $\sqrt{f(t)}$ has the
form claimed and that (iii) holds in the case $n$ is even.\\
(ii) The case for $n$ odd follows similarly since
\[
[\,ht + a_{0};a_{1} , \cdots, a_{n}] =  ht+ A_{n}/B_{n} = ht +
c/h,
\]   by a remark preceding
Lemma~\ref{L:cfinv}.

\end{proof}


\begin{theorem}\label{T:d2m}
Let $f(t)= (c - 1)^{2}h^{2}t^{2} + 2(c - 1)^{2}t + f$.\\
(i)If $n$ is even
\begin{align*}
\sqrt{f(t)}=[\,(c- 1)ht + a_{0};\overline{a_{1}, \cdots,a_{n}
,2(c - 1)ht + 2a_{0}}].
\end{align*}
(ii) If $n$ is odd and $\sqrt{f} =
[\,a_{0};\overline{a_{1},\cdots ,a_{m-1},a_{m},
a_{m-1},\cdots, a_{1},2a_{0}}]\,$, where $a_{m} = a_{0}$ or $a_{0} - 1$
and $m$ is odd then
\begin{align*}
&\sqrt{f(t)}=[(c-1)ht+a_{0};\\
&\phantom{asdafasf}\overline{a_{1},\cdots,a_{m-1},(c-1)ht+a_{m},a_{m-1},\cdots,a_{1},2(c-1)ht+2a_{0}}].
\end{align*}
(iii) In either case $X=(c - 1)h^{4}t^{2} + 2(c - 1)h^{2}t +
c,\,\,\,Y =h^{3}t + h$ constitute the fundamental
solution to $X^{2} - f(t)Y^{2}=1$.
\end{theorem}

\begin{proof}
In (iii) straightforward calculation shows that the given expressions
for $X$ and $Y$ do constitute \emph{a} solution to $X^{2} -
f(t)Y^{2}=1$. What needs to be shown for (iii) is that these choices
of $X$ and $Y$ give the \emph{fundamental} solution.\\
(i)Recall that for $n$ even, $A_{n}^{2}-fB_{n}^{2}=-1$. Notice that
\begin{align*}
&[\,(c- 1)ht + a_{0};a_{1} ,\cdots, a_{n}]=(c- 1)ht+\frac{A_{n}}{B_{n}}=
\frac{(c- 1)htB_{n}+A_{n}}{B_{n}}  \mbox{ and}\\
&((c- 1)htB_{n}+A_{n})^{2}-f(t)B_{n}^{2} = -1,
\end{align*}
by the remark
above, the formula for $f(t)$ and Lemma~\ref{L:ansq}.
Therefore, by Lemma~\ref{L:smasol}, $\sqrt{f(t)}$ has the form
claimed and, by Lemma~\ref{L:ansq}, the smallest solution to $X^{2} -
f(t)Y^{2}=1$ is given by
\begin{align*}
&X=
2((c- 1)htB_{n}+A_{n})^{2}+1 =
(c - 1)h^{4}t^{2} + 2(c - 1)h^{2}t + c,\\
&Y = 2((c- 1)htB_{n}+A_{n})B_{n} = h^{3}t + h.
\end{align*}
(ii)For $n$ odd
{\allowdisplaybreaks
\begin{align*}
&[\,(c- 1)ht + a_{0};a_{1} ,\cdots, a_{m-1},(c - 1)ht +
a_{m},a_{m-1},\cdots, a_{1}]\\
&=[\,(c- 1)ht + a_{0};a_{1},\cdots, a_{m-1},(c - 1)ht
+\frac{B_{m}}{B_{m-1}}], \mbox{ by Lemma~\ref{L:cfinv}}\\
&=(c- 1)ht +\frac{\left((c - 1)ht+\frac{B_{m}}{B_{m-1}}\right)A_{m-1}+A_{m-2}}
        {\left((c -
1)ht+\frac{B_{m}}{B_{m-1}}\right)B_{m-1}+B_{m-2}}\\
&= (c- 1)ht +\frac{(c - 1)htB_{m-1}A_{m-1} + B_{m}B_{m-1}+A_{m-2}B_{m-1}}
        {(c - 1)htB_{m-1}^{2} + B_{m}B_{m-1}+B_{m-2}B_{m-1}}\\
&= (c- 1)ht + \frac{(c-1)h^{2}t+c}{h^{3}t+h}, \,
    \mbox{by Lemmas~\ref{L:sol} and~\ref{L:trbl} and~\eqref{E:recur}}\\
&=\frac{(c - 1)h^{4}t^{2} + 2(c - 1)h^{2}t + c}{h^{3}t + h}.
\end{align*}
}
The results follow by Lemma~\ref{L:smasol}.

\end{proof}


\begin{theorem}\label{T:deg2P}
Let $f(t)= (c + 1)^{2}h^{2}t^{2} + 2(c + 1)^{2}t + f$.\\
If $n$ is odd and $\sqrt{f} =
[\,a_{0};\overline{a_{1},\cdots,a_{m-1},a_{m},
a_{m-1},\cdots, a_{1},2a_{0}}]\,$ , where $a_{m} = a_{0}$ or $a_{0} - 1$
and $m$ is even then
\begin{align*}
&\sqrt{f(t)}=[(c+1)ht+a_{0};\\
&\phantom{asafasdfas}\overline{a_{1},\cdots,a_{m-1},(c + 1)ht+a_{m},
a_{m-1},\cdots,a_{1},2(c+1)ht+2a_{0}}]
\end{align*}
and
$X=(c + 1)h^{4}t^{2} + 2(c + 1)h^{2}t + c,\,\,\,
Y=h^{3}t + h$ constitute the fundamental
solution to $X^{2} - f(t)Y^{2}=1$.
\end{theorem}

\begin{proof}
The proof here is virtually identical to the proof of part (ii) of the
theorem above, the only difference being that part (iii) of
Lemma~\ref{L:trbl} is used instead of part (ii).
\end{proof}


\begin{theorem}\label{T:deg2Pb}
Let $f(t)= (c + 1)^{2}h^{2}t^{2} + 2(c^{2} - 1)t + f$.\\
If $n$ is even then
\begin{align*}
\sqrt{f(t)}=
[\,(c+ 1)ht + a_{0};\overline{a_{1},\cdots,a_{n},2(c + 1)ht + 2a_{0}}]
\end{align*}
and
\[X=\displaystyle{\frac{(c + 1)^{2}}{c-1}}h^{4}t^{2} +
2(c + 1)h^{2}t + c,\,\,\,
Y=
\displaystyle{\frac{(c + 1)}{c-1}}h^{3}t + h\]
 constitute the fundamental
solution to $X^{2} - f(t)Y^{2}=1$.
\end{theorem}
\begin{proof}
By Lemma~\ref{L:ansq} $c = 2A_{n}^{2} + 1$ and $h= 2A_{n}B_{n}$ and
so
$B_{n}^{2}=h^{2}/(2c-2)$. Also, $A_{n}^{2}-fB_{n}^{2}=-1$.
\begin{align*}
[\,(c+ 1)ht + a_{0};a_{1} ,\cdots, a_{n}]&=(c+ 1)ht
+\frac{A_{n}}{B_{n}}=\frac{(c+ 1)htB_{n}+A_{n}}{B_{n}},\\
\mbox{ and } ((c+ 1)htB_{n}+A_{n})^{2}-f(t)B_{n}^{2} &=
2tB_{n}(c+1)(A_{n}h-(c-1)B_{n})-1\\ &=-1,
\end{align*}
\mbox{ by Lemma~\ref{L:ansq}}.
Therefore the continued fraction expansion of $\sqrt{f(t)}$ has the
form claimed and, again by Lemma~\ref{L:ansq},  the fundamental
solution to $X^{2} - f(t)Y^{2}=1$ is given by
\begin{align*}
X&=
2((c+ 1)htB_{n}+A_{n})^{2}+1\\
&= 2(c+ 1)^{2}h^{2}t^{2}B_{n}^{2}+4(c+ 1)htB_{n}A_{n}+2A_{n}^{2}+1,\\
&=\frac{(c + 1)^{2}}{c-1}h^{4}t^{2} + 2(c + 1)h^{2}t + c.\\
Y&=2((c+ 1)htB_{n}+A_{n})B_{n}=2(c+ 1)htB_{n}^{2}+2A_{n}B_{n}=
\frac{(c + 1)}{c-1}h^{3}t + h
\end{align*}

\end{proof}


\begin{theorem}\label{T:deg4o2d}
Let
\[f(t)=(c-1)^{2}h^{6}t^{4} + 4(c-1)^{2}h^{4}t^{3} +
 6(c-1)^{2}h^{2}t^{2} + 2(c-1)(2c-1)t + f.
\]
(i)If $n$ is odd and $\sqrt{f} =
[\,a_{0};\overline{a_{1} ,\cdots ,a_{m-1},a_{m},
a_{m-1},\cdots, a_{1},2a_{0}}]\,$ where  $a_{m} = a_{0}$ or $a_{0} - 1$ and
$m$ is odd
then
\begin{multline*}
\sqrt{f(t)}= [\,(c- 1)(h^{3}t^{2} + 2ht) + a_{0};\\
\overline{a_{1} ,\cdots , a_{m-1},(c - 1)ht +
a_{m},a_{m-1},\cdots,  a_{1}
,2(c - 1)(h^{3}t^{2} + 2ht) + 2a_{0}}]
\end{multline*}
(ii)If $n$ is even then
\begin{multline*}
\sqrt{f(t)}= [\,(c- 1)(h^{3}t^{2} + 2ht) + a_{0};\\
\overline{a_{1} , \cdots, a_{n},2(c - 1)ht +
2a_{0},a_{1}, \cdots, a_{n}
,2(c - 1)(h^{3}t^{2} + 2ht) + 2a_{0}}].
\end{multline*}
(iii) In either case $X=(c-1)h^{6}t^{3} + 3(c-1)h^{4}t^{2} +
  3(c-1)h^{2}t + c,\,\,Y=h + h^{3}t$ constitute the fundamental
  solution to $X^{2}-f(t)Y^{2}=1$.
\end{theorem}

\begin{proof}
(i) As in the proof of Theorem~\ref{T:d2m}
\begin{align*}
&[\,(c- 1)(h^{3}t^{2} + 2ht) + a_{0};
a_{1} ,\cdots , a_{m-1},(c - 1)ht +
a_{m},a_{m-1}, \cdots, a_{1}]\\
&= (c- 1)(h^{3}t^{2} + 2ht) + \frac{(c-1)h^{2}t+c}{h^{3}t+h}\\
&= \frac{(c-1)h^{6}t^{3} + 3(c-1)h^{4}t^{2} +
  3(c-1)h^{2}t + c}{ h^{3}t+h}
\end{align*}
The results follow from Lemma~\ref{L:smasol}.\\
(ii) Similarly,
\begin{align*}
&[\,(c- 1)(h^{3}t^{2} + 2ht) + a_{0};
a_{1} ,\cdots , a_{n},2(c - 1)ht +
2a_{0},a_{1}, \cdots, a_{n}]\\
&= [\,(c- 1)(h^{3}t^{2} + 2ht) + a_{0};
a_{1} ,\cdots , a_{n},2(c - 1)ht +
a_{0}+\frac{A_{n}}{B_{n}}]\\
&= (c- 1)(h^{3}t^{2} + 2ht) +
\frac{ \left(2(c - 1)ht+a_{0}+\frac{A_{n}}{B_{n}}\right)A_{n}+A_{n-1}}
     {\left(2(c - 1)ht+a_{0}+\frac{A_{n}}{B_{n}}\right)B_{n}+B_{n-1}}\\
&= (c- 1)(h^{3}t^{2} + 2ht) +
\frac{\left(2(c - 1)htB_{n}+a_{0}B_{n}+A_{n}\right)A_{n}+A_{n-1}B_{n}}
    {\left(2(c -
1)htB_{n}+a_{0}B_{n}+A_{n}\right)B_{n}+B_{n-1}B_{n}}\\
&= (c- 1)(h^{3}t^{2} + 2ht) +\frac{(c-1)h^{2}t+c}{h^{3}t+h},\,\,
\mbox{ by Lemmas~\ref{L:azeo} and~\ref{L:ansq}}.
\end{align*}
The remainder of the proof parallels part (i).
\end{proof}

\section{Fundamental Units in Real Quadratic Fields}
In many cases it is easy to use the theorems in this paper to write
down the fundamental unit in a wide class of real quadratic
fields. On page 119 of~\cite{N90} one has the following statement of
the relationship of fundamental units in $\mathbb{Q}(\sqrt{D})$, $D$ a
square-free positive rational integer, and the convergents in the
continued fraction expansion for $\sqrt{D}$:
\begin{theorem}\label{T:6}
Let $D$ be a square-free, positive rational integer and let
$K=\mathbb{Q}(\sqrt{D})$. Denote by $\epsilon_{0}$ the fundamental
unit of $K$ which exceeds unity, by $s$ the period of the continued
fraction expansion for $\sqrt{D}$, and by $P/Q$ the ($s-1$)-th
convergent of it.

If $D \not \equiv 1 \mod{4}$ or $D \equiv 1 \mod{8}$, then
\[\epsilon_{0} = P + Q \sqrt{D}.
\]
However, if $D \equiv 5 \mod{8}$, then
\[\epsilon_{0} = P + Q \sqrt{D}.
\]
or
\[\epsilon_{0}^{3} = P + Q \sqrt{D}.
\]
Finally, the norm of $ \epsilon_{0}$ is positive if the period $s$ is
even and negative otherwise.
\end{theorem}
We note first of all that the congruence class of $f$,  mod 4,
forces $c$ and $h$ to lie in certain congruence classes. We have the following table:
\begin{table}[ht]
  \begin{center}
    \begin{tabular}{| c | c | c | }
    \hline
    $f$ & $c$ & $h$  \\ \hline
    1 (mod 4) & $\pm$ 1 (mod 8) &  0 (mod 4)\\ \hline
    2 (mod 4) & $\pm$ 1 (mod 16) &  0 (mod 4) \\
     & $\pm$ 3 (mod 8) &  2 (mod 4) \\\hline
       3 (mod 4)    &  $\pm$ 1 (mod 8)     &  0 (mod 4) \\
         & 0 (mod 2)      & 1 (mod 2) \\ \hline
    \end{tabular}
\phantom{asdf}\\
    \caption{}\label{Ta:t1}
    \end{center}
\end{table}

 In
what follows let $f$ be a non-square positive integer and  $(c,h)$ be the
smallest pair of positive integers satisfying $c^{2}-fh^{2}=1$.
We denote the Fermat-Pell polynomials from Theorems \ref{T:hto} - \ref{T:deg4o2d}
as follows:
{\allowdisplaybreaks
\begin{align}\label{cases1}
f_{1}(t)&=h^{2}t^{2} + 2 c t + f, \\
f_{2}(t)&=(c - 1)^{2}h^{2}t^{2} + 2(c - 1)^{2}t + f, \notag \\
f_{3}(t)&=(c + 1)^{2}h^{2}t^{2} + 2(c + 1)^{2}t + f, \notag\\
f_{4}(t)&=(c + 1)^{2}h^{2}t^{2} + 2(c^{2} - 1)t + f, \notag\\
f_{5}(t)&=(c-1)^{2}h^{2}t^{2}(h^{5}t^{2}+4h^{2}t+6) + 2(c-1)(2c-1)t + f. \notag
\end{align}
}
We consider first the case $f\equiv 2 \mod{4}$.  From the table above,
$c$ has to be odd, $h$ has to be even and thus $f_{2}(t)$, $f_{3}(t)$,
$f_{4}(t)$ and $f_{5}(t)$
 are  $\equiv 2
\mod{4}$, for all integral $t \geq 0$ and
$f_{1}(t)$ is $\equiv 2 \mod{4}$ for all \emph{even}  $t \geq 0$.
Thus we can apply Theorem \ref{T:6}.

As an illustration,
by Theorems \ref{T:hto} and \ref{T:6}, if $f(t)=4h^{2}t^{2} + 4 c t + f$ is
squarefree, then the fundamental unit in $\mathbb{Q}( \sqrt{f(t)})$ is
$2h^{2}t+c + h\sqrt{f(t)}$. With $f=22$, $c = 197$ and $h = 42$,
\,$22 + 788\,t + 7056\,t^2$ is squarefree for $150,601$ of the values of $t$
lying between 0 and 200,000.  In particlar, it is squarefree for
$t = 199,998$ so that the fundamental unit in $\mathbb{Q}(\sqrt{282234512826670})$
is immediately known to be $705593141 + 42 \sqrt{282234512826670}$.

Remark: The data in this case suggests the following question:
\[\text{Is }\lim_{N->\infty}
\frac{\#\{t \in \mathbb{N}, t\leq N:  22 + 788\,t + 7056\,t^2
\text{ is squarefree} \}}{N} =\frac{ 3}{4}?
\]
The author is presently unable to provide the answer.

A  similar situation holds for either of the two possibilities listed in
Table \ref{Ta:t1} for the case $f\equiv 3 \mod{4}$. In the first case
($c\equiv \pm 1 \mod{8}$ and $h\equiv 0 \mod{4}$), $f_{2}(t)$, $f_{3}(t)$,
$f_{4}(t)$ and $f_{5}(t)$
take positive integral values which are $\equiv 3 \mod{4}$
for all integral $t \geq 0$, while $f_{1}(t)$ takes on
 positive integral values which are $\equiv 3 \mod{4}$
for all integral \emph{even} $t \geq 0$. In the second case
($c\equiv 0 \mod{2}$ and $h\equiv 1 \mod{2}$), all of the polynomials
take positive integral values which are $\equiv 3 \mod{4}$
for all \emph{even} integral $t \geq 0$.
Theorem \ref{T:6} can again be used to
write down the fundamental unit in infinitely many real quadratic fields.
We again consider $h^{2}t^{2} + 2 c t + f$, with $t$ even, and take
 $f=43$, $c=3482$ and $h=531$. It is found that
$43 + 13928\,t + 1127844\,t^2$
 is squarefree for 147,511 of the  integers $t$ lying between
0 and 200,000. In particlar, it is squarefree for
$t = 199,999$ so that the fundamental unit in
$\mathbb{Q}(\sqrt{45113311649113959})$
is immediately known to be $112783839560 + 531 \sqrt{45113311649113959}$.

Lastly, we consider the case $f\equiv 1 \mod{4}$.  For
$c \equiv 1 \mod{8}$, $f_{2}(t)$, $f_{3}(t)$,
$f_{4}(t)$ and $f_{5}(t)$ $\equiv f \mod{8}$, for all integral $t \geq 0$,
while $f_{1}(t) \equiv f \mod{8}$ for all integral
$t\equiv 0 \mod{4}$.
For  $c \equiv -1 \mod{8}$,
$f_{2}(t)$, $f_{3}(t)$ and $f_{4}(t)$
   $\equiv f \mod{8}$, for all integral $t \geq 0$,
 $ f_{1}(t)\equiv f \mod{8}$ for all positive
$t\equiv 0 \mod{4}$ and $f_{5}(t)$  $\equiv f \mod{8}$ for all
\emph{even} $t \geq 0$.
For the case $c \equiv  0$ (mod 2) and $h \equiv  1$ (mod 2),
$ f_{1}(t)$ and $ f_{5}(t)\equiv f \mod{8}$ for all positive
$t\equiv 0 \mod{4}$, $f_{2}(t)$ $f_{3}(t)$ and
$f_{4}(t)\equiv f \mod{8}$ for all positive
\emph{even} $t$.
 In all of these cases, if $f \equiv 1 \mod{8}$
and the resulting polynomial is squarefree, then the fundamental unit in the
corresponding quadratic field can be written down. This time we consider
the polynomial $f_{2}(t) = (c - 1)^{2}h^{2}t^{2} + 2(c - 1)^{2}t + f$.
By Theorems \ref{T:d2m} and \ref{T:6},  if the continued fraction expansion
of $\sqrt{f}$ has even period and $f_{2}(t)$ is
squarefree, then the fundamental unit in $\mathbb{Q}( \sqrt{f(t)})$ is
$(c - 1)h^{4}t^{2} + 2(c - 1)h^{2}t +c + (h^{3}t + h)\sqrt{f_{2}(t)}$.
As an illustration we take $f=57$ so that $c= 151$ and $h=20$.
$f(t)$ is found to be squarefree for
121,529  of the  integers $t$ lying between
0 and 130,000. In particular, it is squarefree for $t = 130,000$ so
so that the fundamental unit in $\mathbb{Q}(\sqrt{152100005850000057})$
is immediately known to be
$ 405600015600000151 +1040000020 \sqrt{152100005850000057}$.

\section{Concluding Remarks}
In this paper only some limited classes of
Fermat-Pell polynomials were considered
 (for example in all cases all but the end and
possibly the middle terms in the continued fraction expansion were
constant and the highest degree considered was $4$). In a later paper I
will consider multi-variable Fermat-Pell polynomials, polynomials whose
continued fraction expansion has all terms non-constant and
single-variable Fermat-Pell polynomials where the degree can be
arbitrarily large.

Some problems still remain for the classes of
polynomials examined here. Each of the triple of polynomials examined
here constitute a polynomial solution to Pell's by virtue of the fact
that $c^{2} - fh^{2}=1$ and this alone. A natural question is to
determine the continued fraction expansion of $f(t)$ in each of the
cases examined when $(c,h)$, (the smallest pair of positive integers
satisfying the above integral Pell's equation) is replaced by the
$n$th largest pair of such integers.

The continued fraction expansion of some of polynomials examined
were given in some cases only in special circumstances. For example
the continued fraction expansion of
$f(t)=(c + 1)^{2}h^{2}t^{2} + 2(c + 1)^{2}t + f$, examined in
Theorem~\ref{T:deg2P}, was given only in the case where
 \[
\sqrt{f} =
[\,a_{0};\overline{a_{1} ,\cdots ,a_{m-1},a_{m},
a_{m-1},\cdots, a_{1},2a_{0}}]\,\text{ where } a_{m} = a_{0}
\text{ or } a_{0} - 1
\]
and $m$ is even.
There remains the problem of determining the continued fraction
expansion of $f(t)$ when the length of the period of the
continued fraction expansion of  $\sqrt{f}$ is
$ \not \equiv 0 \mod 4$ or $a_{m} \not = a_{0}$
 or  $a_{0} - 1$. Similar problems
remain with the continued fraction expansion of the polynomials
examined in some of the other theorems.

\end{document}